\documentclass[11pt,a4paper,oneside]{amsart}
\title{The special linear group for nonassociative rings}
\author{Harry Petyt}

\usepackage[T1]{fontenc} 
\usepackage[utf8]{inputenc} 
\usepackage[english]{babel} 
\usepackage{amsmath} 
\usepackage{mathtools} 
\usepackage{amssymb} 
\usepackage{amsthm} 
\usepackage{csquotes} 
\usepackage{accents} 
\usepackage{tikz-cd} 
\usepackage[a4paper]{geometry} 
\usepackage{setspace} 
\usepackage{graphicx} 
\usepackage[nodayofweek]{datetime} 
\usepackage[toc,page]{appendix} 
\usepackage{float} 
\usepackage[font=small,justification=centering]{caption} 
\usepackage{empheq} 
\usepackage{xfrac} 

\newcounter{claimcount}

\newtheorem{theorem}{Theorem}[section]
\newtheorem{lemma}[theorem]{Lemma}
\newtheorem{corollary}[theorem]{Corollary}

\theoremstyle{definition}
\newtheorem{definition}[theorem]{Definition}

\theoremstyle{remark}
\newtheorem{remark}[theorem]{Remark}
\newenvironment{claim}[1]{\stepcounter{claimcount}\par\noindent\underline{Claim:}\space#1}{}
\newenvironment{claimproof}[1]{\par\noindent\underline{Proof:}\space#1}{\leavevmode\unskip\penalty9999 \hbox{}\nobreak\hfill\quad\hbox{$\diamondsuit$}\vspace{2mm}}

\newcommand*{\C}{\mathbb{C}}
\newcommand*{\F}{\mathbb{F}}
\newcommand*{\HH}{\mathbb{H}}
\newcommand*{\mfh}{\mathfrak{h}}

\newcommand*{\mfM}{\mathfrak{M}}

\newcommand*{\OO}{\mathbb{O}}

\newcommand*{\R}{\mathbb{R}}
\newcommand*{\Z}{\mathbb{Z}}
\newcommand*{\msl}{\mathfrak{sl}}

\newcommand{\vcc}{\hspace{1mm}\vcentcolon\hspace{1mm}}

\DeclareMathOperator{\tr}{tr}
\DeclareMathOperator{\real}{Re}
\DeclareMathOperator{\SL}{SL}
\DeclareMathOperator{\End}{End}

\begin{document}

\maketitle

\begin{abstract}
We extend to arbitrary rings a definition of the octonion special linear group due to Baez. At the infinitesimal level we get a Lie ring, which we describe over some large classes of rings, including all associative rings and all algebras over a field. As a corollary we compute all the groups Baez defined.
\end{abstract}

\section{Introduction}

The special linear groups $\SL_2(\R)$ and $\SL_2(\C)$ are, respectively, the double covers of $\mathrm{SO}_0(2,1)$ and $\mathrm{SO}_0(3,1)$ -- the isometry groups of the hyperbolic plane and hyperbolic 3-space. The pattern continues with the quaternions $\HH$, as shown by Kugo and Townsend in \cite{kugotownsend}, and Sudbery deals with the final normed real division algebra, the octonions $\OO$, in \cite{sudbery}. Unfortunately the way Sudbery defines the special linear group over $\OO$ only makes sense in dimensions two and three. 

In his celebrated survey \cite{baez}, Baez suggests a unified definition of $\SL_m(\OO)$ for all $m$, and shows that it agrees with Sudbery's definition when $m=2$. He does not discuss the case $m>2$, and it seems that until now no further investigation has been made. 

Motivated by this, in Section 2 we reformulate Baez's definition of the special linear group and algebra in a natural way that lends itself to computation, and note that it naturally extends to arbitrary nonassociative rings (in the present paper, we do not in general assume rings to be associative). We then determine the corresponding special linear ring (we do not necessarily get an algebra structure) for all associative rings. In Section 3 we cover the two dimensional case for unital real composition algebras. In Section 4 we characterise $\SL_m$, with $m>2$, for a large class of algebras that includes $\OO$. This allows us to compute Baez's groups. In doing so, we find that in three dimensions his definition disagrees with Sudbery's, which gives a real form of the exceptional Lie group $E_6$. 

An alternative definition for $\SL_2(\OO)$ has been proposed by Hitchin \cite{hitchin}. This definition is motivated by a dimension argument, and does not give a Lie group.

I would like to thank Dmitriy Rumynin for introducing me to the problem and for his helpful comments and suggestions.

\section{Preliminaries}

\subsection{}

A \emph{composition algebra} is a not necessarily unital or associative algebra $C$ over a field $\F$, together with a nondegenerate quadratic form $|\cdot|^2$ that is multiplicative in the sense that $|zw|^2=|z|^2|w|^2$. Such algebras come with an anti-involution, which we call \emph{conjugation} and denote by a bar, e.g. $\bar z$. They are also necessarily alternative. That is, the \emph{associator} $[\cdot,\cdot,\cdot]:C^3\rightarrow C$, given by $[z,w,u]=(zw)u-z(wu)$, is alternating.

If the characteristic of $\F$ is not two, then all unital composition algebras can be obtained from $\F$ by the Cayley--Dickson construction (a description of which can be found in \cite{schafer}), and have famously been classified by Jacobson \cite{jacobson}. We are mainly interested in real composition algebras, and we state his classification in this case.

\begin{theorem}[Jacobson]\label{composition}
The unital real composition algebras are exactly: \\
\begin{tabular}{ll}
(i) $\R$ & \\
(ii) $\C$ & (iii) $\R^2$, with quadratic form $|(a,b)|^2=ab$ \\
(iv) $\HH$  & (v) \hspace{1mm}$M_2(\R)$, the $2\times2$ matrices over $\R$, with $|\cdot|^2=\det$ \\
(vi) $\OO$ \hspace{10mm} & (vii) The split octonions $\OO'$
\end{tabular}
\end{theorem}

We write $1,e_1,\dots,e_{d-1}$ for an orthonormal basis of a unital composition algebra of dimension $d$. We then have $e_i^2=\pm1$ for all $i$, and $e_ie_j=-e_je_i$ whenever $i\neq j$. Letting $L_z:C\rightarrow C$ denote the left multiplication map $w\mapsto zw$, alternativity of $C$ gives us 
\begin{align}
L_{e_i}L_{e_i}=L_{e_i^2}=\pm1  	\label{squaresto}
\end{align}
Moreover, whenever $e_i\neq e_j$ we have
\begin{align*}
0=[e_i,e_j,z]+[e_j,e_i,z] & =L_{e_ie_j+e_je_i}(z)-(L_{e_i}L_{e_j}+L_{e_j}L_{e_i})(z)
\end{align*}
and hence
\begin{align}
L_{e_i}L_{e_j}=-L_{e_j}L_{e_i} \label{anticommutativity}
\end{align}

Let $\R^{p,q}$ denote $\R^{p+q}$ with the standard quadratic form of signature $(p,q)$. The algebras in the left-hand column of Theorem~\ref{composition} have signature $(d,0)$, and those in the right-hand column have signature $(\frac{d}{2},\frac{d}{2})$.

\subsection{}

For a ring $R$ we write $M_m(R)$ to mean the space of $m\times m$ matrices with entries in~$R$, and $E_{ij}$ for an element of the standard basis. The trace of a matrix $x$ is written $\tr x$, and left multiplication maps are again denoted $L_x$. The definition of the octonion special linear group and algebra given by Baez is as follows \cite[p.177]{baez}.

\begin{definition}[Baez]\label{baezdefinition}
The octonion special linear algebra $\msl_m(\OO)$ is the Lie algebra generated under commutators by the set $\left\{L_x : x\in M_m(\OO), \hspace{1mm}\tr x=0 \right\}$. The octonion special linear group $\SL_m(\OO)$ is the Lie group generated by exponentiating $\msl_m(\OO)$.
\end{definition}

This definition is not well suited to computation, and we prefer to use the following, which is easily seen to agree with Definition~\ref{baezdefinition} in the case $R=\OO$.

\begin{definition}\label{generaldefinition}
For $R$ a not necessarily associative or unital ring, $\msl_m(R)$ is the ring generated by $\left\{L_{aE_{ij}} : a\in R, \hspace{1mm} i\neq j \right\}$ under commutators. Similarly, $\SL_m(R)$ is the group generated by $\left\{L_{I+aE_{ij}} : a\in R, \hspace{1mm}i\neq j \right\}$ under composition.
\end{definition}

Straight from the definition we can obtain a nice description of the special linear algebra of an associative ring.

\begin{theorem}\label{associativecase}
Let $R$ be an associative (not necessarily unital) ring. Then there is an isomorphism $\msl_m(R)\cong\lbrace x\in M_m(R) \hspace{1mm} \vcentcolon \hspace{1mm} \tr x\in \left[ R,R \right] \rbrace$
\end{theorem}

\begin{proof}
Since $R$ is associative, we have $L_aL_b=L_{ab}$ for all $a,b\in R$, so we can identify $L_{aE_{ij}}$ with the matrix $aE_{ij}$ and consider $\msl_m(R)\subset M_m(R)$. Now $\msl_m(R)$ contains all matrices with all diagonal entries zero, as these form the linear span of the generators. Furthermore, the commutator of two generators is 	$\left[ aE_{ij},bE_{kl} \right] = \delta_{jk}abE_{il}-\delta_{il}baE_{kj}$, where $\delta$ denotes the Kronecker delta. If $\delta_{jk}$ and $\delta_{il}$ are not both $1$ then we get either zero or a generator. If both are 1 then we get $abE_{ii}-baE_{jj}$. Clearly this has trace lying in $\left[ R,R\right]$, and by varying $a$ and $b$ we can get the whole of $[R,R]$. Then varying $i$ and $j$ gives the right hand side of the result. Note that the commutator of such a diagonal matrix with a generator is traceless, so all further commutators have trace in $[R,R]$, and we are done.
\end{proof}

Theorem \ref{associativecase} shows that Definition~\ref{generaldefinition} gives a true generalisation of the usual special linear algebra, for if $R$ is a field then $[R,R]=0$, and, moreover, if $R=\HH$ then $[\HH,\HH]=\{z\in\HH:\real(z)=0\}$, which gives the standard definition of $\msl_m(\HH)$ \cite[p.52]{harvey}.

\section{The Two Dimensional Case}

Let $C$ be a unital real composition algebra. For $x=(x_{ij})\in M_m(C)$, the \emph{hermitian conjugate} of $x$ is $x^*=(\overline{x_{ji}})$. If $x^*=x$ then $x$ is said to be hermitian, and the set of such matrices is denoted $\mfh_m(C)$. Note that all diagonal entries of a hermitian matrix lie in $\R$. We restrict our attention to the case $m=2$, where alternativity of $C$ ensures that the determinant map $x\mapsto x_{11}x_{22}-x_{12}\overline{x_{12}}$ is a well defined quadratic form on $\mfh_2(C)$ (see \cite[p.176]{baez}). 

If $C$ has dimension $d$ and signature $(p,q)$, then writing $z=z_0+z_1e_1+\dots+z_{d-1}e_{d-1}$ for an element of $C$, we have that $\mfh_2(C)$ is isometric to $\R^{q+1,p+1}$ via the map 
\[
\left( 
\begin{array}{cc}
r & z \\
\bar{z} & s \end{array}\right) \longmapsto 
\left(\frac{r+s}{2}, \frac{r-s}{2}, z_0, z_1,\dots,z_{d-1}\right)
\]

We now define a representation of $\SL_2(C)$ on $\mfh_2(C)$. Let $y=L_{I+aE_{ij}}$ be a generator of $\SL_2(C)$, and for $x\in\mfh_2(C)$ set $y\cdot x=(I+aE_{ij})x(I+\bar{a}E_{ji})$. This product is well defined because $C$ is alternative, and we extend to $\SL_2(C)$ in the obvious way.

\begin{lemma}
The action of $\SL_2(C)$ on $\mfh_2(C)$ is by isometries. That is, if $x\in\mfh_2(C)$ and $y\in\SL_2(C)$ then $\det(y\cdot x)=\det x$.
\end{lemma}

\begin{proof}
It suffices to show that this holds for generators of $\SL_2(C)$. The two cases are similar, so we just do $y=L_{I+aE_{21}}$. Let $x=\left(\begin{array}{cc}r&z\\\bar{z}&s\end{array}\right)$, recalling that $r,s\in\R$. Since $w\bar w=\bar ww$ for all $w\in C$ we obtain
\[
\det(y\cdot x) 
\hspace{1mm}=\hspace{1mm} \det\left( 
    \begin{array}{cc}
    r & r\bar{a}+z \\
    ra+\bar{z} & ra\bar{a}+\bar{z}\bar{a}+az+s \end{array}\right) 
\hspace{1mm}=\hspace{1mm} rs-\bar{z}z \hspace{1mm}=\hspace{1mm} \det x 
\qedhere\]
\end{proof}

It follows that there is a homomorphism of connected Lie groups
\[
\psi\vcc \SL_2(C)\longrightarrow \mathrm{SO}_0(q+1,p+1)
\]
In order to analyse $\psi$, we describe a basis of $\msl_2(C)$, but first a remark.

\begin{remark}\label{overcomplex}
If $\F=\C$ then the same argument as the one above gives a homomorphism $\SL_2(\C)\rightarrow\mathrm{SO}(d+2,\C)$.
\end{remark}

\begin{lemma}\label{basislemma}
$\msl_2(C)$ is based by the set
\[
\Big\{ L_{E_{12}}, \hspace{0.5mm} L_{E_{21}}, \hspace{0.5mm} [L_{E_{12}},L_{E_{21}}], \hspace{0.5mm} \alpha_i=L_{e_iE_{12}}, \hspace{0.5mm} \beta_i=L_{e_iE_{21}}, \hspace{0.5mm} \gamma_i=[L_{E_{12}},\beta_i], \hspace{0.5mm} \varepsilon_{ij}=[\alpha_i,\beta_j] : i<j\Big\}
\] 
In particular, $\dim\SL_2(C)=3+3(d-1)+\frac{(d-1)(d-2)}{2}=\frac{(d+1)(d+2)}{2}$.
\end{lemma}

\begin{proof}
It follows from identities \eqref{squaresto} and \eqref{anticommutativity} that the set in question bases the subspace spanned by products of length at most two, so it suffices to show that this is the whole of $\msl_2(C)$. Products of length three are spanned by generators and elements $\delta$ and $\delta^T$, where \[
\delta=\left(\begin{array}{cc}
0 & L_{e_i}L_{e_j}L_{e_k}+L_{e_k}L_{e_j}L_{e_i} \\
0 & 0 \end{array} \right)
\]
There are three cases for $\delta$, depending on the choice of $e_i,e_j,$ and $e_k$. \\[1mm]
\underline{Case 1:} $e_j$ is equal to either $e_i$ or $e_k$. Then $\delta$ is a generator by identity~\eqref{squaresto}. \\[1mm]
\underline{Case 2:} $e_i=e_k\neq e_j$. Using both identities \eqref{squaresto} and \eqref{anticommutativity} we see that $\delta$ is a generator. \\[1mm]
\underline{Case 3:} $e_i, e_j, e_k$ are distinct. Then $\delta=0$ by identity \eqref{anticommutativity}. \\[1mm]
Thus products of length three are spanned by generators, which completes the proof.
\end{proof}

\begin{lemma}\label{trivialkernel}
$\ker d\psi=0$
\end{lemma}
\begin{proof}
The action of $\SL_2(C)$ on $\mfh_2(C)$ induces an action of $\msl_2(C)$: if $x\in\mfh_2(C)$ and $y$ is a generator of $\msl_2(C)$ then $y\cdot x=yx+xy^*$. By definition, any element of $\ker d\psi$ acts trivially, so we calculate the action of the basis of Lemma~\ref{basislemma} on an arbitrary $x=\left(\begin{array}{cc}r&z\\\bar{z}&s\end{array}\right)\in\mfh_2(C)$. Here $r,s\in\R$, $z=z_0+\sum_{i=1}^{d-1}z_ie_i\in C$, and below the $\lambda_i$ and $\kappa_{ij}$ are real. In several places we find it convenient to write $w=\lambda_0+\sum_{i=1}^{d-1}\lambda_ie_i$.
\begin{flalign*}
(a): \Big(\lambda_0L_{E_{12}}+\sum_{i=1}^{d-1}\lambda_i\alpha_i\Big)\cdot x 
	\hspace{2mm} &= \hspace{2mm} L_{wE_{12}}\cdot x\hspace{2mm} =\hspace{2mm} \left(\begin{array}{cc}2\real(w\bar{z})&sw\\s\bar{w}&0 \end{array}\right) &\\
(b): \Big(\lambda_0L_{E_{21}}+\sum_{i=1}^{d-1}\lambda_i\beta_i\Big)\cdot x
	\hspace{2.65mm}&=\hspace{2mm}L_{wE_{21}}\cdot x\hspace{2mm}=\hspace{2mm}\left(\begin{array}{cc} 0&r\bar{w}\\rw&2\real(wz)\end{array}\right) &
\end{flalign*}
\begin{flalign*}
(c): \Big(\lambda_0[L_{E_{12}},L_{E_{21}}]+\sum_{i=1}^{d-1}\lambda_i&\gamma_i\Big)\cdot x
	\hspace{2mm}=\hspace{2mm}[L_{E_{12}},L_{wE_{21}}]\cdot x &\\
		&=L_{E_{12}}\cdot\left(\begin{array}{cc} 0&r\bar{w}\\rw&2\real(wz)\end{array}\right)
		-L_{wE_{21}}\cdot\left(\begin{array}{cc}2\real(\bar{z})&s\\s&0 \end{array}\right) &\\
		&=2\left(\begin{array}{cc} r\lambda_0&wz-z\bar{w}\\
		\bar{z}\bar{w}-w\bar{z}&-s\lambda_0\end{array}\right) &
\end{flalign*}
\begin{flalign*}
(d): \Big(\sum_{0<i<j<d}&\kappa_{ij}\varepsilon_{ij}\Big)\cdot x
		\hspace{2mm}=\hspace{2mm}\sum_{i<j}\kappa_{ij}\left(L_{e_iE_{12}}\cdot(L_{e_jE_{21}}\cdot x)
		-L_{e_jE_{21}}\cdot(L_{e_iE_{12}}\cdot x)\right) &\\
	&\hspace{-1cm}=\sum_{i<j}\kappa_{ij}\bigg(\left(\begin{array}{cc}2\real(e_ire_j)&2\real(e_jz)e_i\\
		2\real(e_jz)\bar{e_i}&0 \end{array}\right) 
		-\left(\begin{array}{cc} 0&2\real(e_i\bar{z})\bar{e_j}\\
		2\real(e_i\bar{z})e_j&2\real(e_jse_i)\end{array}\right)\bigg) &\\
	&\hspace{-1cm}=\sum_{i<j}2\kappa_{ij}\left(\begin{array}{cc} 0&e_j^2z_je_i-e_i^2z_ie_j\\
		-e_j^2z_je_i+e_i^2z_ie_j&0 \end{array}\right) &
\end{flalign*}
Together, these describe the action of every element of $\msl_2(C)$ on $\mfh_2(C)$. Now assume that $y$ acts trivially. Writing $y$ in the basis of Lemma \ref{basislemma} and using the four equations above (with respective sets of coefficients $\lambda_i$, $\mu_i$, $\nu_i$, and $\kappa_{ij}$), consider $y\cdot x$. 

The upper-left entry is $2\real\big((\lambda_0+\sum_{i=1}^{d-1}\lambda_ie_i)\bar{z}\big)+2r\nu_0$, which must be zero for all $r\in\R$ and $z\in C$. Taking $z=0$, $r=1$ gives $\nu_0=0$, and then cycling $z$ through the $e_i$ gives $\lambda_i=0$. Similarly, considering the lower right entry gives $2\real\big((\mu_0+\sum\mu_ie_i)z\big)-2s\nu_0=0$. We already have $\nu_0=0$, and cycling $z$ through the $e_i$ gives $\mu_i=0$. Finally, considering the top right entry, we are left with
\[
\sum_{i=1}^{d-1}2(\nu_ie_iz-z\nu_i\bar{e_i})+\sum_{0<i<j<d}2\kappa_{ij}(e_j^2z_je_i-e_i^2z_ie_j)=0
\]
Taking $z=1$ gives $\sum_1^{d-1}4\nu_ie_i=0$, so all $\nu_i$ are zero. Then successively considering $z=e_1, z=e_2, \dots$ we find that all $\kappa_{1j}, \kappa_{2j},\dots$ are zero. Thus $\ker d\psi=0$.
\end{proof}

Knowing that $d\psi$ has full rank is enough to prove the main theorem of this section.

\begin{theorem}\label{theorem2d}
If $C$ is a unital real composition algebra of dimension $d$ and signature $(p,q)$ then $\SL_2(C)\cong \mathrm{Spin}(p+1,q+1)$.
\end{theorem}

\begin{proof}
Because Spin$(p+1,q+1)\cong$ Spin$(q+1,p+1)$ it suffices to show that $\psi$ is onto and has two-point kernel. By Lemma \ref{basislemma} we have $\dim\SL_2(C)=\frac{(d+1)(d+2)}{2}=\dim\mathrm{SO}_0(q+1,p+1)$, so by Lemma~\ref{trivialkernel} $d\psi$ is onto. Hence $\psi$ is onto. Indeed,
\[
\psi(\SL_2(C))=\exp(d\psi(\msl_2(C)))=\exp(\mathfrak{so}(q+1,p+1))=\mathrm{SO}_0(q+1,p+1)
\]

It remains to show that $\psi$ has two-point kernel. Consider the real matrices
\[
a=\left(\begin{array}{cc}1&-1\\ 0&1\end{array}\right),
\hspace{5mm}b=\left(\begin{array}{cc}1&0 \\1&1\end{array}\right),
\hspace{5mm}c=\left(\begin{array}{cc}1&-2\\0&1\end{array}\right)
\]
and the linear map $\iota=L_aL_bL_cL_bL_a\in\SL_2(C)$, which acts as $-I$ on $C^2$. Because $a,b,c\in M_2(\R)$, the expression for $\iota\cdot x$ associates, so $\iota\cdot x=(-I)x(-I)=x$, and $\iota\in\ker\psi$. Hence $\ker\psi$ consists of at least two elements.
	
\begin{claim}{}
If $C$ is associative then $\ker\psi=\{1,\iota\}$.
\end{claim}
\begin{claimproof}{}
As in the proof of Theorem \ref{associativecase} we can consider elements of $\SL_2(C)$ to be matrices. Then $\iota=-I$. Also, if $y=(y_{ij})\in\SL_2(C)$ acts trivially on $\mfh_2(C)$ we have 
\[
\left(\begin{array}{cc}r&0 \\ 0&s\end{array}\right)=
y\cdot\left(\begin{array}{cc}r&0 \\0&s\end{array}\right) 
=\left(\begin{array}{cc}		ry_{11}\overline{y_{11}}+sy_{12}\overline{y_{12}} \hspace{3mm}
&ry_{11}\overline{y_{21}}+sy_{12}\overline{y_{22}}\\
ry_{21}\overline{y_{11}}+sy_{22}\overline{y_{12}} \hspace{3mm}
&ry_{21}\overline{y_{21}}+sy_{22}\overline{y_{22}}
\end{array}\right)	
\]
Taking $r=1, s=0$ in this gives $y_{11}\overline{y_{11}}=1$ and $y_{21}=0$. Similarly, taking $r=0, s=1$ gives $y_{22}\overline{y_{22}}=1$ and $y_{12}=0$. Now we have 
\[
\left(\begin{array}{cc} 0&z\\\bar{z}&0\end{array}\right)
    =\left(\begin{array}{cc}y_{11}&0 \\ 0&y_{22}\end{array}\right)
        \cdot\left(\begin{array}{cc}0&z\\\bar{z}&0\end{array}\right)
    =\left(\begin{array}{cc} 0& y_{11}z\overline{y_{22}} \\ y_{22}\overline{z}\hspace{0.2mm}\overline{y_{11}} & 0
\end{array}\right)
\]
Taking $z=1$ gives $y_{11}\overline{y_{22}}=1$. But $y_{11}\overline{y_{11}}=1$, so $y_{22}=y_{11}$. From this we get $z=y_{11}z\overline{y_{11}}$, so $zy_{11}=y_{11}z$ for all $z\in C$. Hence $y_{11}=y_{22}=\pm1$, and thus $y\in\{1,\iota\}$.
\end{claimproof}
	
For the case $C=~\OO$, note that since $\pi_1(\mathrm{SO}(9,1))=\Z_2$ \cite[pp.335,~343]{hall}, any proper cover is a double cover.
	
This just leaves the case $C=\OO'$. The complexification $\C\otimes\OO'$ is isomorphic to the bioctonions $\C\otimes\OO$, a unital complex composition algebra. Thus $\SL_2(\OO')$ is a real form of $\SL_2(\C\otimes\OO)$, which covers $\mathrm{SO}(10,\C)$ by Remark~\ref{overcomplex}. Since $\pi_1(\mathrm{SO}(10,\C))=\Z_2$ \cite[p.343]{hall}, the covering is 2:1, with kernel $\{1,\iota\}$. We  thus have a commutative diagram
\[
\begin{tikzcd}[column sep=huge]
\SL_2(\OO') \arrow[r, "\mathrm{complexify}"] \arrow[d, "\psi"]
& \SL_2(\C\otimes\OO) \arrow[d, "2:1"] \\
	\mathrm{SO}_0(5,5) \arrow[r, "\mathrm{complexify}"] 
& \mathrm{SO}(10,\C)
\end{tikzcd}
\]
which completes the proof in the final case, $C=\OO'$.
\end{proof}

\section{The General Case}

Let $R$ be a commutative associative unital ring, and $A$ a finite dimensional $R$--algebra that is free as an $R$--module, with basis $\{e_1,\dots,e_n\}$. For example, $A$ could be any finite dimensional algebra over a field. 

Write $\mfM_A$ for the left multiplication algebra of $A$. That is, $\mfM_A$ is the $R$--algebra generated by $\{L_{re_i}:r\in R\}$. Since $A$ is free over $R$ we have $\mfM_A\subset\End_R(R^n)=M_n(R)$.

\begin{theorem}\label{mgeq3theorem}
Under the above assumptions, if $m\geq3$ then $\msl_m(A)\cong\big\{x\in M_m(\mfM_A) \vcc \tr x\in [\mfM_A,\mfM_A] \big\}$. In particular $\msl_m(A)\subset\msl_{mn}(R)$.
\end{theorem}

\begin{proof}
The nonzero products of two generators are:
\[
[L_{aE_{ij}},L_{bE_{ji}}]=L_{L_aL_bE_{ii}}-L_{L_bL_aE_{jj}}
\]
\[
[L_{aE_{ij}},L_{bE_{jk}}]=L_{L_aL_bE_{ik}}
\]
Since $m\geq3$, it follows by taking successive products that $\msl_m(A)$ is the span of the set 
\[
\Big\{ L_{\alpha E_{ij}} \vcc \alpha\in \mfM_A,\textrm{ } i\neq j \Big\}
\bigcup
\Big\{ L_{\alpha\beta E_{ii}}-L_{\beta\alpha E_{jj}} \vcc \alpha,\beta\in \mfM_A, \textrm{ } i\neq j \Big\}
\]
Clearly all such matrices have trace in $[\mfM_A,\mfM_A]$, and varying $\alpha$ and $\beta$ gives the whole of $[\mfM_A,\mfM_A]$. Varying $i$ and $j$ then gives the result.
\end{proof}

This reduces the problem of determining $\msl_m(A)$ to that of finding $\mfM_A$. In the case of the $\R$-algebra $\OO$, the group generated by left multiplications by units is $\mathrm{SO}(8)$ \cite[p.92]{conwaysmith}. This $\R$-spans the full matrix algebra $M_8(\R)$, so $\mfM_\OO=M_8(\R)$. We can thus calculate Baez's groups:

\begin{corollary}\label{baezgroups}
	If $m\geq3$ then $\SL_m(\OO)\cong\SL_{8m}(\R)$.
\end{corollary}

\begin{proof}
Theorem~\ref{mgeq3theorem} gives $\msl_m(\OO)\cong\msl_{8m}(\R)$. Exponentiating gives the result.
\end{proof}

Together with Theorem \ref{theorem2d} this describes all the groups Baez defined. In fact, the same argument works for $\OO'$, and we similarly obtain $\SL_m(\OO')\cong\SL_{8m}(\R)$ for $m\geq3$.

\begin{remark}
Corollary~\ref{baezgroups} shows that Baez's definition of $\SL_3(\OO)$ disagrees with Sudbery's, which gives a real form of $E_6$. For $m=2,3$, Sudbery defines $\msl_m(\OO)$ to be the left multiplications by traceless elements of $\mfh_m(\OO)$, together with its derivations. The isomorphism with $\mathfrak{e}_6$ is due to Chevalley and Schafer \cite{chevalleyschafer}. It is natural to ask what happens to this construction when $n$ is greater than 3. In this case, it is shown in \cite[Thm.~3.3]{petyt} that the derivation algebra of $\mfh_m(\OO)$ is $\mathfrak{g}_2\oplus\mathfrak{so}_m$, but it is not clear how this interacts with the multiplication operators. In particular, when $m>3$ the commutator of two such operators may fail to be a derivation.  
\end{remark}

\bibliographystyle{alpha}
\bibliography{bibtex}

\end{document}